\documentclass[12pt]{article}
\usepackage{amsmath,amsthm, amssymb}
\usepackage{amssymb,latexsym}

\newtheorem{theorem}{Theorem}

\newtheorem{lemma}{Lemma}

\newcommand{\bp}{\begin{proof}[Proof:]}
\newcommand{\ep}{\end{proof}}

\newcommand{\ds}{\displaystyle}

\newcommand{\Obig}{\mathcal{O}}
\newcommand{\Ob}[1]{\mathcal{O}\left(#1\right)}

\begin{document}

\baselineskip=17pt

\title{\bf Diophantine approximation by special primes}

\author{\bf S. I. Dimitrov}
\date{2017}
\maketitle
\begin{abstract}
We show that whenever $\delta>0$, $\eta$ is real and constants $\lambda _i$ satisfy some necessary
conditions, there are infinitely many prime triples $p_1,\, p_2,\, p_3$ satisfying the inequality
$|\lambda _1p_1 + \lambda _2p_2 + \lambda _3p_3+\eta|<(\max p_j)^{-1/12+\delta}$
and such that, for each $i\in\{1,2,3\}$, $p_i+2$ has at most $28$ prime factors.
\medskip

{\bf  2000 Math.\ Subject Classification}.  11D75, 11N36, 11P32.
\par
\end{abstract}
\section{Introduction and statements of the result.}
\indent

In 1973 Vaughan \cite{Vaughan} proved that whenever $\delta>0$, $\eta$ is real and constants $\lambda_i$ satisfy some necessary
conditions, there are infinitely many prime triples $p_1,\,p_2,\,p_3$ such that
\begin{equation}\label{1}
|\lambda_1p_1+\lambda_2p_2+\lambda_3p_3+\eta|<(\max p_j)^{-\xi+\delta}
\end{equation}
for $\xi=1/10$. Latter the upper bound for $\xi $ was improved by Baker and Harman \cite{Baker-Harman} to $\xi=1/6$,
by Harman \cite{Harman} to $\xi=1/5$
and the best result up to now is due to K. Matom\"{a}ki \cite{Mato2} with $\xi=2/9$.

On the other hand a famous and still unsolved problem in Number Theory is the prime-twins
conjecture, which states that there exist infinitely many prime numbers
$p$ such that $p+2$ is also a prime.

Up to now many hybrid theorems were proved. One of the best result belongs to K. Matom\"{a}ki
and Shao \cite{Mato1}.
They proved that every sufficiently large odd integer $n$ such that $n\equiv 3\,\pmod {6}$ can be represented  as a sum
\begin{equation*}
n=p_1+p_2+p_3
\end{equation*}
of primes $p_1,\,p_2,\,p_3$ such that
\begin{equation*}
p_1+2=P'_2\,, \quad p_2+2=P''_2\,,  \quad p_3+2=P'''_2\,,
\end{equation*}
where $P_l$ is a number with at most $l$ prime factors.

In the present paper we consider (\ref{1}) with primes of the form specified above. We prove the following theorem.

\begin{theorem}
Suppose that $\lambda_1,\lambda_2,\lambda_3$ are non-zero real numbers, not all of
the same sign, that $\eta$ is real, and that $\lambda_1/\lambda_2$ is irrational.
Let $\xi=1/12$ and $\delta>0$.
Then there are infinitely many ordered triples of primes $p_1,\,p_2,\,p_3$  for which

\begin{equation}\label{Theorem}
 |\lambda_1p_1+\lambda_2p_2+\lambda_3p_3+\eta|<(\max p_j)^{-\xi+\delta}
\end{equation}
and
\[
p_1+2=P'_{28}\,, \quad p_2+2=P''_{28}\,, \quad  p_3+2=P'''_{28}\,.
\]
\end{theorem}
By choosing the parameters in a different way we may obtain other similar results,
for example $\xi=-9/350\,,\,p_i+2=P_{20}\,,i=1,2,3$.

Result of this type were obtained by  Dimitrov and Todorova \cite{DimTod-1}.
Combining the circle and sieve methods and using the Bombieri -- Vinogradov prime number theorem
they proved (\ref{Theorem}) with right-hand side $[\log(\max p_j)]^{-A}$, $A>1$ and primes
$p_1,\,p_2,\,p_3$ such that $p_i+2=P_8\,,i=1,2,3$.
In this paper we improve the right-hand side of \cite{DimTod-1}.
Obviously this is at the expense of the number of the prime factors of $p_i+2$.

\section{Notations and some lemmas.}
\indent

For positive $A$ and $B$ we
write $A\asymp B$ instead of $A\ll B\ll A$.
As usual $\varphi(n)$ and $\mu (n)$
denote  Euler's function and M\"{o}bius' function.
Let $(m_1,m_2)$ and $[m_1,m_2]$ be
the greatest common divisor and the least common multiple
of $m_1,m_2$ respectively. Instead of $m\equiv n\,\pmod {k}$
we write for simplicity $m\equiv n(k)$.  As usual, $[y]$ denotes
the integer part of $y$, $e(y)=e^{2\pi \imath y}$. The letter
$\varepsilon$ denotes an arbitrary small positive number, not the
same in all appearances. For example this convention allows us to
write $x^{\varepsilon }\log x\ll x^{\varepsilon }$.
Since $\lambda_1/\lambda_2$ is irrational, there are infinitely many different convergents
$a_0/q_0$ to its continued fraction, with
\begin{equation}\label{q0}
  \bigg|\frac{\lambda_1}{\lambda_2}- \frac{a_0}{q_0}\bigg|<\frac{1}{q_0^2}
\end{equation}
where $(a_0, q_0) = 1, q_0 \geq1 $ and $a_0\neq0$ .
We choose $q_0$ to be large in terms of $\lambda_1,\lambda_2,\lambda_3$ and  $\eta$,
and make the following definitions.

\begin{align}
\label{X}
&X=q_0^{12/5}\,;\\
\label{tau}
&\tau=X^{-5/6}\log X\,;\\
\label{eps}
&\vartheta=X^{-1/12+\delta}\,,\;\;\delta>0\,;\\
\label{H}
&H=\frac{\log^2X}{\vartheta}\,;\\
\label{z}
&z=X^\beta\,,\;\;0<\beta<1/30\,;\\
\label{D}
&D=X^{47/450-\varepsilon_0}\,,\;\;\varepsilon_0=0.001\,;\\
\label{notPz}
&P(z)=\prod\limits_{2<p\le z}p\,,\,\ \mbox{$ p$ -prime number}\,;\\
\label{I}
&I(\alpha)=\int\limits_{\lambda_0X}^{X}e(\alpha y)dy\,.
\end{align}
The value of $\beta$  will be specified latter.

Let $\lambda ^{\pm}(d)$ be the lower and upper bounds Rosser's weights of level
$D$, hence
\begin{equation}\label{Rosser}
|\lambda ^{\pm}(d)|\leq1\,,\quad\lambda ^{\pm}(d)=0\;\;  \mbox{  if }\;\; d \geq D \;\;\mbox{ or }\;\; \mu (d)=0\,.
\end{equation}
For further properties of Rosser's weights we refer to \cite{Iwa1}, \cite{Iwa2}.
\begin{lemma}\label{Fourier}Let $\vartheta\in \mathbb{R}$ and $k\in \mathbb{N}$.
There exists a function $\theta(y)$ which is $k$ times continuously differentiable and
such that
\begin{align*}
  &\theta(y)=1\quad\quad\quad\mbox{for }\quad\quad|y|\leq 3\vartheta/4\,;\\[6pt]
  &0\leq\theta(y)<1\quad\mbox{for}\quad3\vartheta/4 <|y|< \vartheta\,;\\[6pt]
  &\theta(y)=0\quad\quad\quad\mbox{for}\quad\quad|y|\geq \vartheta\,.
\end{align*}
and its Fourier transform
\begin{equation*}
\Theta(x)=\int\limits_{-\infty}^{\infty}\theta(y)e(-xy)dy
\end{equation*}
satisfies the inequality
\begin{equation*}
|\Theta(x)|\leq\min\bigg(\frac{7\vartheta}{4},\frac{1}{\pi|x|},\frac{1}{\pi |x|}
\bigg(\frac{k}{2\pi |x|\vartheta/8}\bigg)^k\bigg)\,.
\end{equation*}
\end{lemma}
\begin{proof} See \cite{Segal}.
\end{proof}

\begin{lemma}\label{log}
Let  $X\geq2$, $k\geq2$. We have
\begin{equation*}
 \sum\limits_{n\leq X}\frac{1}{\varphi(n)}\ll\log X\,.
\end{equation*}
\end{lemma}

\section{Outline of the proof.}
\indent

Consider the sum
\begin{equation}\label{Gamma}
\Gamma(X)= \sum\limits_{\lambda_0X<p_1,p_2,p_3\leq X\atop{|\lambda_1p_1+\lambda_2p_2+\lambda_3p_3+\eta|< \vartheta\atop{(p_{i}+2,P(z))=1,i=1,2,3}}}\log p_{1} \log p_{2}\log p_{3}\,.
\end{equation}
Any non-trivial estimate from below of $\Gamma(X)$ implies solvability of
$|\lambda_1p_1+\lambda_2p_2+\lambda_3p_3+\eta|<\vartheta$ in primes such that $p_{i}+2=P_h,\;h=[\beta^{-1}]$.

We have
\begin{equation}\label{Gammaest1}
\Gamma(X)\geq\widetilde{\Gamma}(X)=\sum\limits_{\lambda_0X<p_1,p_2,p_3\leq X\atop{(p_{i}+2,P(z))=1,i=1,2,3}}
\theta(\lambda_1p_1+\lambda_2p_2+\lambda_3p_3+\eta)\log p_1 \log p_2\log p_3\,.
\end{equation}
On the other hand
\begin{equation}\label{Gammawave}
\widetilde{\Gamma}(X)=\sum\limits_{\lambda_0X<p_1,p_2,p_3\leq X}\theta(\lambda_1p_1+\lambda_2p_2+\lambda_3p_3+\eta)
\Lambda_{1}\Lambda_{2}\Lambda_{3}\log p_1 \log p_2\log p_3\,,
\end{equation}
where
\begin{equation*}
\Lambda_i=\sum\limits_{d|(p_{i}+2,P(z))}\mu(d)\,,\,i=1,2,3
\end{equation*}
We denote
\begin{equation}\label{Lambdaipm}
\Lambda_i^{\pm}=  \sum\limits_{d|(p_{i}+2,P(z))}\lambda^{\pm}(d)\,,\,i=1,2,3
\end{equation}
From the linear sieve we know that $\Lambda_{i}^{-}\le \Lambda_{i}\le \Lambda_{i}^{+}$
(see \cite{Bru}, Lemma 10).
Then we have a simple inequality
\begin{equation}\label{Inequality}
\Lambda_1\Lambda_2\Lambda_3\ge \Lambda_1^{-} \Lambda_2^{+} \Lambda_3^{+}+
\Lambda_{1}^{+} \Lambda_{2}^{-} \Lambda_{3}^{+}+\Lambda_{1}^{+}\Lambda_{2}^{+}\Lambda_{3}^{-}-
2\Lambda_{1}^{+}\Lambda_{2}^{+} \Lambda_{3}^{+}
\end{equation}
analogous to this one in (\cite{Bru}, Lemma 13).\\
Using (\ref{Gammawave}) and (\ref{Inequality}) we obtain
\begin{align}\label{Gammawaveest}
\widetilde{\Gamma}(X)&\geq\Gamma_{0}(X)=\sum\limits_{\lambda_0X<p_1,p_2,p_3\leq X}
\theta(\lambda_1p_1+\lambda_2p_2+\lambda_3p_3+\eta)\nonumber\\
&\times(\Lambda_{1}^{-} \Lambda_{2}^{+} \Lambda_{3}^{+}+\Lambda_{1}^{+}\Lambda_{2}^{-}\Lambda_{3}^{+}
+\Lambda_{1}^{+}\Lambda_{2}^{+} \Lambda_{3}^{-}-2\Lambda_{1}^{+} \Lambda_{2}^{+} \Lambda_{3}^{+})
\log p_1\log p_2\log p_3\,.
\end{align}
Let
\begin{equation}\label{Gamma0}
\Gamma_{0}(X)=\Gamma_{1}(X)+\Gamma_{2}(X)+\Gamma_{3}(X)-2\Gamma_{4}(X)\,,
\end{equation}
where for example
\begin{equation}\label{Gamma1}
\Gamma_1(X)=\sum\limits_{\lambda_0X<p_1,p_2,p_3\leq X}\theta(\lambda_1p_1+\lambda_2p_2+\lambda_3p_3+\eta)
\Lambda_1^- \Lambda_2^+\Lambda_3^+\log p_1\log p_2\log p_3
\end{equation}
and so on.
We shall consider the sum $\Gamma_1(X)$. The rest can be considered in the same way.

From \eqref{Lambdaipm} and \eqref{Gamma1} we have
\begin{equation*}
\Gamma_1(X)=  \sum\limits_{d_{i}|P(z)\atop{i=1,2,3}}\lambda^-(d_1)\lambda^+(d_2)\lambda^+(d_3)
\sum\limits_{\lambda_0X<p_1,p_2,p_3\leq X\atop{p_{i}+2\equiv0\,(d_i), i=1,2,3}}
\theta(\lambda_1p_1+\lambda_2p_2+\lambda_3p_3+\eta)\log p_1\log p_2\log p_3.
\end{equation*}

Using the inverse Fourier transform for the function $\theta(x)$ we get
\begin{align*}
\Gamma_{1}(X)&=\sum\limits_{d_{i}|P(z)\atop{i=1,2,3}}\lambda^{-}(d_1)\lambda^{+}(d_2)\lambda^{+}(d_3)
\sum\limits_{\lambda_0X<p_1,p_2,p_3\leq X\atop{p_{i}+2\equiv0\,(d_{i}), i=1,2,3}}\log p_{1} \log p_{2}\log p_{3}\\
&\times\int\limits_{-\infty}^{\infty}\Theta(t)e((\lambda_1p_1+\lambda_2p_2+\lambda_3p_3+\eta)t)dt\\
&=\int\limits_{-\infty}^{\infty}\Theta(t)e(\eta t)L^-(\lambda_1t,X)L^+(\lambda_2t,X)L^+(\lambda_3t,X)dt\,,
\end{align*}
where
\begin{equation}\label{L}
L^\pm(t,X)=\sum\limits_{d|P(z)}\lambda^\pm(d)\sum\limits_{\lambda_0X<p\leq X\atop{p+2\equiv0\,(d)}}e(pt)\log p\,.
\end{equation}
We divide $\Gamma_{1}(X)$ into  three parts
\begin{equation}\label{Gama1}
\Gamma_1(X)=\Gamma_1^{(1)}(X)+\Gamma_1^{(2)}(X)+\Gamma_1^{(3)}(X)\,,
\end{equation}
where
\begin{equation}\label{Gamma1,1}
\Gamma_1^{(1)}(X)=\int\limits_{|t|\leq\tau}\Theta(t)e(\eta t)L^-(\lambda_1t,X)L^+(\lambda_2t,X)L^+(\lambda_3t,X)dt\,,
\end{equation}
\begin{equation}\label{Gamma1,2}
\;\;\;\;\Gamma_1^{(2)}(X)=\int\limits_{\tau<|t|<H}\Theta(t)e(\eta t)L^-(\lambda_1t,X)L^+(\lambda_2t,X)L^+(\lambda_3t,X)dt\,,
\end{equation}
\begin{equation}\label{Gamma1,3}
\Gamma_1^{(3)}(X)=\int\limits_{|t|\geq H}\Theta(t)e(\eta t)L^-(\lambda_1t,X)L^+(\lambda_2t,X)L^+(\lambda_3t,X)dt\,.
\end{equation}
We shall estimate $\Gamma_1^{(3)}(X)$, $\Gamma_1^{(1)}(X)$, $\Gamma_1^{(2)}(X)$ respectively in the sections 4, 5, 6.
In section 7 we shall complete the proof of the Theorem.
\section{Upper bound for $\mathbf{\Gamma_{1}^{(3)}(X)}$.}
\begin{lemma} For the integral $\Gamma_{1}^{(3)}(X)$, defined by (\ref{Gamma1,3}), we have
\begin{equation}\label{Gamma1,3est}
\Gamma_{1}^{(3)}(X)\ll1\,.
\end{equation}
\end{lemma}
\begin{proof} See [\cite{DimTod-1}, Lemma 2].
\end{proof}
\section{Asymptotic formula for $\mathbf{\Gamma_{1}^{(1)}(X)}$.}
\indent

The first lemma we need in this section is the following.
\begin{lemma}\label{intLintI}Let $\lambda\neq0$.
Using the definitions (\ref{I}) and (\ref{L}) we have
\begin{align*}
&\emph{(i)}\quad\quad\quad\int\limits_{-\tau}^\tau|L^\pm(\lambda\alpha,X)|^2d\alpha\ll X\log^5X\,,\\
&\emph{(ii)}\quad\quad\quad\int\limits_{-\tau}^\tau|I(\lambda\alpha)|^2d\alpha\ll X\log X\,.
\end{align*}
\end{lemma}
\begin{proof} We only prove (i). The inequality (ii) can be proved
likewise.

Having in mind  (\ref{tau}), (\ref{Rosser}) and (\ref{L})  we get
\begin{align}
\label{estIntSd2}
\int\limits_{-\tau}^\tau|L^\pm(\lambda\alpha,X))|^2d\alpha&=
\sum\limits_{d_i|P(z)\atop{i=1,2}}\lambda ^{\pm}(d_1)\lambda ^{\pm}(d_2)\nonumber\\
&\times\sum\limits_{\lambda_0X<p_1,p_2\leq X\atop{p_1+2\equiv0\,(d_1)
\atop{p_2+2\equiv0\,(d_2)}}}\log p_1\log p_2
\int\limits_{-\tau}^\tau e(\lambda(p_1-p_2)\alpha)d\alpha\nonumber\\
&\ll\sum\limits_{d_i\leq D\atop{i=1,2}}\sum\limits_{\lambda_0X<p_1,p_2\leq
X\atop{p_1+2\equiv 0\,(d_1)\atop{p_2+2\equiv 0\,(d_2)}}}\log p_1\log p_2
\min\bigg(\tau,\frac{1}{|p_1-p_2|}\bigg)\nonumber\\
&\ll\sum\limits_{d_i\leq D\atop{i=1,2}}\left(\tau\sum\limits_{\lambda_0X<p_1,p_2\leq
X\atop{p_1+2\equiv 0\,(d_1)\atop{p_2+2\equiv 0\,(d_2)\atop{|p_1-p_2|\leq1/\tau}}}}
\log p_1\log p_2+\sum\limits_{\lambda_0X<p_1,p_2\leq
X\atop{p_1+2\equiv 0\,(d_1)\atop{p_2+2\equiv 0\,(d_2)\atop{|p_1-p_2|>1/\tau}}}}
\frac{\log p_1\log p_2}{|p_1-p_2|}\right)\nonumber\\
&\ll(\log X)^2\sum\limits_{d_i\leq D\atop{i=1,2}}(U\tau+V)\,,
\end{align}
where
\begin{equation*}
U=\sum\limits_{\lambda_0X<n_1,n_2\leq
X\atop{n_1+2\equiv 0\,(d_1)\atop{n_2+2\equiv 0\,(d_2)\atop{|n_1-n_2|\leq1/\tau}}}}1\,,\quad
V=\sum\limits_{\lambda_0X<n_1,n_2\leq
X\atop{n_1+2\equiv 0\,(d_1)\atop{n_2+2\equiv 0\,(d_2)\atop{|n_1-n_2|>1/\tau}}}}\frac{1}{|n_1-n_2|}\,.
\end{equation*}
We have
\begin{align}
\label{USd}
U&\ll\mathop{\sum\limits_{\lambda_0X<n_1\leq
X\atop{n_1+2\equiv 0\,(d_1)}}\sum\limits_{\lambda_0X<n_2\leq X\atop{n_2+2\equiv 0\,(d_2)}}1}_
{n_1-1/\tau\leq n_2\leq n_1+1/\tau}
\ll\sum\limits_{\lambda_0X<n_1\leq X\atop{n_1+2\equiv 0\,(d_1)}}
\bigg(\frac{1}{\tau d_2}+1\bigg)\nonumber\\
&\ll\frac{1}{\tau d_2}\sum\limits_{\lambda_0X<n_1\leq X\atop{n_1+2\equiv 0\,(d_1)}}1
\ll\frac{X}{\tau d_1d_2}\,.
\end{align}

Obviously  $V\leq\sum\limits_{l}V_l$  where
\begin{equation}\label{VlSd}
V_l=\sum\limits_{\lambda_0X<n_1,n_2\leq
X\atop{l<|n_1-n_2|\leq2l\atop{n_1+2\equiv 0\,(d_1)\atop{n_2+2\equiv 0\,(d_2)}}}}\frac{1}{|n_1-n_2|}
\end{equation}
and $l$ takes the values $2^d/\tau,\,d=0,1,2,...$, with $l\leq X$.

We have
\begin{align}
\label{estVlSd}
V_l&\ll\frac{1}{l}\mathop{\sum\limits_{\lambda_0X<n_1\leq
X\atop{n_1+2\equiv 0\,(d_1)}}\sum\limits_{\lambda_0X<n_2\leq X\atop{n_2+2\equiv 0\,(d_2)}}1}_
{n_1+l\leq n_2\leq n_1+2l}
\ll\frac{1}{l}\sum\limits_{\lambda_0X<n_1\leq X\atop{n_1+2\equiv 0\,(d_1)}}\bigg(\frac{l}{d_2}+1\bigg)\nonumber\\
&\ll\frac{1}{d_2}\sum\limits_{\lambda_0X<n_1\leq X\atop{n_1+2\equiv0\,(d_1)}}1
\ll\frac{X}{d_1d_2}\,.
\end{align}
The assertion in (i) follows from (\ref{tau}), (\ref{estIntSd2}) -- (\ref{estVlSd}).
\end{proof}

The next lemma gives us asymptotic formula for the sums $L^\pm(\alpha,X)$ denoted by (\ref{L}).
\begin{lemma}\label{LAsympt} Let $D$ is defined by (\ref{D}),
and $\lambda(d)$ be complex numbers defined for $d\le D$ such that
\begin{equation}\label{omega}
|\lambda(d)|\leq1\,,\quad\lambda(d)=0\;\;\mbox{ if }\;\;2|d\;\;\mbox{ or }\;\;\mu(d)=0\,.
\end{equation}
If
\begin{equation*}
L(\alpha,X)=\sum\limits_{d\le D}\lambda(d)
\sum\limits_{\substack{ \lambda_0X<p\le X \\ p+2\equiv0\,(d)}}e(p\alpha)\log p
\end{equation*}
then for $|\alpha|\le\tau$ we have
\begin{equation}\label{Lasympt}
L(\alpha,X)=I(\alpha)\sum\limits_{d\leq D}\frac{\lambda(d)}{\varphi(d)}
+\mathcal{O}\bigg(\frac{X}{(\log X)^A}\bigg)\,,
\end{equation}
where $A>0$ is an arbitrary large constant.
\end{lemma}
\begin{proof}
This lemma is very similar to results of Tolev  \cite{Tolev}.
Inspecting the arguments presented in (\cite{Tolev}, Lemma 10),
the reader will easily see that the proof of Lemma 5
can be obtained by the same manner.
\end{proof}
Let
\begin{align}\label{Mipm}
&L^\pm_i=L^\pm(\lambda_it,X)\,,\nonumber\\
&\mathcal{M}^\pm_i=\mathcal{M}^\pm(\lambda_it,X)
=I(\lambda_it)\sum\limits_{d\leq D}\frac{\lambda^\pm(d)}{\varphi(d)}\,.
\end{align}
We use the identity
\begin{align}\label{Identity}
L^\pm_1L^\pm_2L^\pm_3&=\mathcal{M}^\pm_1\mathcal{M}^\pm_2\mathcal{M}^\pm_3
+(L^\pm_1-\mathcal{M}^\pm_1)\mathcal{M}^\pm_2\mathcal{M}^\pm_3\nonumber\\
&+L^\pm_1(L^\pm_2-\mathcal{M}^\pm_2)\mathcal{M}^\pm_3+L^\pm_1L^\pm_2(L^\pm_3-\mathcal{M}^\pm_3)\,.
\end{align}
Replace
\begin{equation}\label{J1}
J_1=\int\limits_{|t|\leq\tau}\Theta(t)e(\eta t)\mathcal{M}^-(\lambda_1t,X)
\mathcal{M}^+(\lambda_2t,X)\mathcal{M}^+(\lambda_3t,X)dt\,.
\end{equation}
Then from  (\ref{L}), (\ref{Gamma1,1}), (\ref{Mipm}), (\ref{Identity}), (\ref{J1}),
Lemma \ref{Fourier} and Lemma \ref{LAsympt} we obtain
\begin{align}\label{Gamma1,1J1}
\Gamma_1^{(1)}(X)-J_1
&=\int\limits_{|t|\leq\tau}\Theta(t)e(\eta t)\Big(L^-(\lambda_1t,X)-\mathcal{M}^-(\lambda_1t,X)\Big)
\mathcal{M}^+(\lambda_2t,X)\mathcal{M}^+(\lambda_3t,X)dt\nonumber\\
&+\int\limits_{|t|\leq\tau}\Theta(t)e(\eta t)L^-(\lambda_1t,X)
\Big(L^-(\lambda_2t,X)-\mathcal{M}^-(\lambda_2t,X)\Big)\mathcal{M}^+(\lambda_3t,X)dt\nonumber\\
&+\int\limits_{|t|\leq\tau}\Theta(t)e(\eta t)L^-(\lambda_1t,X)
L^+(\lambda_2t,X)\Big(L^+(\lambda_3t,X)-\mathcal{M}^+(\lambda_3t,X)\Big)dt\nonumber\\
&\ll\vartheta\frac{X}{(\log X)^A}\Bigg(\int\limits_{|t|\leq\tau}
|\mathcal{M}^+(\lambda_2t,X)\mathcal{M}^+(\lambda_3t,X)|dt\nonumber\\
&+\int\limits_{|t|\leq\tau}|L^-(\lambda_1t,X)\mathcal{M}^+(\lambda_3t,X)|dt
+\int\limits_{|t|\leq\tau}|L^-(\lambda_1t,X)L^+(\lambda_2t,X)|dt\Bigg)\nonumber\\
&\ll\vartheta\frac{X}{(\log X)^A}\Bigg(\int\limits_{|t|\leq\tau}
|\mathcal{M}^+(\lambda_2t,X)|^2dt+\int\limits_{|t|<\tau}|\mathcal{M}^+(\lambda_3t,X)|^2dt\nonumber\\
&+\int\limits_{|t|\leq\tau}|L^-(\lambda_1t,X)|^2dt+\int\limits_{|t|\leq\tau}|L^+(\lambda_2t,X)|^2dt\Bigg)\,.
\end{align}
On the other hand (\ref{Mipm}) and Lemma \ref{log} give us
\begin{equation}\label{Mipmest}
|\mathcal{M}^\pm(\lambda_it,X)|\ll|I(\lambda_it)|\log X\,.
\end{equation}
Bearing in mind (\ref{Gamma1,1J1}), (\ref{Mipmest}) and Lemma \ref{intLintI} we find
\begin{equation}\label{Gama1,1J1est}
\Gamma_1^{(1)}(X)-J_1\ll\vartheta \frac{X^2}{(\log X)^{A-5}}\,.
\end{equation}
Arguing as in \cite{DimTod-1} for the integral defined by (\ref{J1}) we get
\begin{equation}\label{J1est}
J_1=B(X)\left(\sum\limits_{d|P(z)}\frac{\lambda ^{-}(d)}{\varphi(d)}\right)
\left(\sum\limits_{d|P(z)}\frac{\lambda ^{+}(d)}{\varphi(d)}\right)^2
+\mathcal{O}(\vartheta \tau^{-2}\log^3X)\,,
\end{equation}
where
\begin{equation*}
B(X)=\int\limits_{\lambda_0X}^{X}\int\limits_{\lambda_0X}^{X}\int\limits_{\lambda_0X}^{X}
\theta(\lambda_1y_1+\lambda_2y_2+\lambda_3y_3+\eta)dy_1dy_2dy_3\,.
\end{equation*}
According to (\cite{DimTod-1}, Lemma 4) we have
\begin{equation}\label{estB}
B(X)\gg\vartheta X^2.
\end{equation}
Usung  (\ref{tau}), (\ref{eps}),  (\ref{Gama1,1J1est}) and (\ref{J1est}) we obtain
\begin{equation}\label{Gam1,1}
\Gamma_1^{(1)}(X)=B(X)\left(\sum\limits_{d|P(z)}\frac{\lambda ^{-}(d)}{\varphi(d)}\right)
\left(\sum\limits_{d|P(z)}\frac{\lambda ^{+}(d)}{\varphi(d)}\right)^2
+\mathcal{O}\left(\vartheta \frac{X^2}{(\log X)^{A-5}}\right)\,.
\end{equation}
Let
\begin{equation}\label{Gpm}
G^\pm=\sum\limits_{d|P(z)}\frac{\lambda^{\pm}(d)}{\varphi(d)}\,.
\end{equation}
Thus from  (\ref{Gam1,1}) and (\ref{Gpm})  it follows
\begin{equation}\label{Ga1,1}
\Gamma_1^{(1)}(X)=B(X)G^-(G^+)^2+\mathcal{O}\left(\vartheta \frac{X^2}{(\log X)^{A-5}}\right)\,.
\end{equation}
\section{Upper bound for $\mathbf{\Gamma_{1}^{(2)}(X)}$.}
\indent

The treatment of the intermediate region depends on the following lemma.
\begin{lemma}\label{TodTolev1}
Suppose $\alpha \in \mathbb{R}$
with a rational approximation $\dfrac{a}{q}$ satisfying
$\ds \bigg|\alpha- \frac{a}{q}\bigg|<\frac{1}{q^2}$,
where  $(a, q) = 1, q \geq 1$, $a\neq0$. Let $D$ is defined by (\ref{D}),
and $\omega(d)$ be complex numbers defined for $d\le
D$ and let $\omega(d)\ll1$.
If
\begin{equation}\label{sn1}
   \mathfrak{L}(X)= \sum\limits_{d\le D}\omega(d)
        \sum\limits_{\substack{ X/2<p\le X \\ p+2\equiv 0\,(d)}}e(p\alpha)\log p
\end{equation}
then
\begin{equation*}
    \mathfrak{L}(X)\ll X^{\varepsilon}\bigg(X^{11/12}+\frac{X}{q^{1/2}}+
    X^{1/2}q^{1/2}+q\bigg)\,,
\end{equation*}
where $\varepsilon$ is an arbitrary small positive number.
\end{lemma}
\begin{proof} See [\cite{TodTolev1}, Lemma 1].
\end{proof}

Let us consider any sum $L^{\pm}(\alpha,\,X)$ denoted by (\ref{L}).
We represent it as sum of finite number
sums of the type
\begin{equation*}
  L(\alpha,\,Y)=\sum\limits_{d\le D}\omega(d)
  \sum\limits_{Y/2<p\leq Y\atop{p+2\equiv0\,(d)}}e(p\alpha)\log p\,,
\end{equation*}
where
\begin{equation*}
  \omega(d)=\left\{
           \begin{array}{ll}
             \lambda^{\pm}(d), & \hbox{if}\quad d|\,P(z)\,, \\
             0, & \hbox{otherwise.}
           \end{array}
         \right.
\end{equation*}
We have
\begin{equation*}
  L^{\pm}(\alpha,\,X)\ll \max\limits_{\lambda _0X\le Y\le X}\left|L(\alpha,\,Y)\right|\,.
\end{equation*}
If
\begin{equation}\label{Int_q}
  q\in\left[X^{1/6},\,X^{5/6}\right]\,,
\end{equation}
then from  Lemma \ref{TodTolev1} for the sums $L(\alpha,\,Y)$ we get
\begin{equation*}
  L(\alpha,\,Y)\ll Y^{11/12+\varepsilon}\,.
\end{equation*}
Therefore
\begin{equation}\label{L_i_alpha}
L^{\pm}(\alpha,\,X)\ll \max\limits_{\lambda _0X\le Y\le X}Y^{11/12+\varepsilon}\ll X^{11/12+\varepsilon}\,.
\end{equation}
Let
\begin{equation}\label{VtX}
V(t,\,X)=\min\left\{\left|L^{\pm}(\lambda_{1}t,\,X)\right|,\left|L^{\pm}(\lambda_2 t,\,X)\right|\right\}\,.
\end{equation}
We shall prove the following

\begin{lemma} Let $t,\,X,\,\lambda_1,\,\lambda_2\in \mathbb{R}$,
\begin{equation}\label{t_delta_H}
  |t|\in (\tau ,\,H)\,,
\end{equation}
where $\tau$ and $H$ are denoted by
(\ref{tau}) and (\ref{H}), $\lambda_1/\lambda_2 \in \mathbb{R}\backslash \mathbb{Q}$
 and $V(t,\,X)$ is defined by (\ref{VtX}).
Then there exists a sequence of real numbers
$X_1,\,X_2,\ldots \to \infty $ such that
\begin{equation}\label{Vtxj}
V(t\,,X_j)\ll X_j^{11/12+\varepsilon}\,,\quad j=1,2,\dots\,.
\end{equation}
\end{lemma}
\begin{proof} Our aim is to prove that
there exists a sequence $X_1,\,X_2,\,...\to \infty $ such that
for each $j=1,2,\ldots $
at least one of the numbers
$\lambda_{1}t$ and $\lambda_{2}t$ with t, subject to (\ref{t_delta_H})
can be approximated by
rational numbers with denominators, satisfying (\ref{Int_q}).
Then the proof follows from (\ref{L_i_alpha}) and (\ref{VtX}).

Let $q_0$ be sufficiently large and $X$ be such that
$X=q_0^{12/5}$ (see (\ref{X})).
Let us notice that there exist $a_1,\,q_1\in \mathbb{Z}$,
such that
\begin{equation}\label{lambda1_a1_q1}
  \bigg|\lambda_1t-\frac{a_1}{q_1}\bigg|<\frac{1}{q_1q_0^2}\,,
  \quad\quad (a_1,\,q_1)=1,\quad\quad 1\leq q_1\leq q_0^2,\quad\quad a_1\ne 0\,.
\end{equation}
From Dirichlet's Theorem (\cite{Karat}, p.158) it follows
the existence
of integers $a_1$ and $q_1$, satisfying the first three conditions.
If $a_1=0$ then $|\lambda_1t|<\ds \frac{1}{q_1q_0^2}$
and from (\ref{t_delta_H}) it follows
\begin{equation*}
 \lambda_1\tau< \lambda_1|t|< \frac{1}{q_0^2}\,,\quad\quad
 q_0^2< \frac{1}{\lambda_1\tau}\;.
\end{equation*}
From the last inequality,  (\ref{X}) and (\ref{tau}) we obtain
\begin{equation*}
  X^{5/6}<\frac{X^{5/6}}{\lambda_1\log X}\,,
\end{equation*}
which is impossible for large $q_0$, respectively, for a large $X$.
So $a_1\ne 0$.
By analogy there exist $a_2,\,q_2\in \mathbb{Z}$,
such that
\begin{equation}\label{lambda2_a2_q2}
  \bigg|\lambda_2t-\frac{a_2}{q_2}\bigg|<\frac{1}{q_2q_0^2}\,,
  \quad\quad (a_2,\,q_2)=1,\quad\quad 1\leq q_2\leq q_0^2,\quad\quad a_2\ne 0\,.
\end{equation}
If $q_i\in\bigg[X^{1/6},\,X^{5/6}\bigg]$
for $i=1$ or $i=2$, then the proof is completed.
From (\ref{X}), (\ref{lambda1_a1_q1}) and
(\ref{lambda2_a2_q2}) we have
\begin{equation*}
  q_i\le X^{5/6}=q_0^2\,,\quad i=1,2\,.
\end{equation*}
Thus it remains to prove that the case
\begin{equation}\label{qi_logx}
  q_i<X^{1/6}\,,\quad i=1,2\,
\end{equation}
is impossible.
Let $q_i<X^{1/6}$, $i=1,2$.
From (\ref{H}), (\ref{t_delta_H}), (\ref{lambda1_a1_q1}) -- (\ref{qi_logx})
it follows
\begin{align}
  & 1\le |a_i|<\frac{1}{q_0^2}+q_i\lambda_i|t|< \frac{1}{q_0^2}+q_i\lambda_i H\,,\nonumber\\
\label{ai}
& 1\le |a_i|<\frac{1}{q_0^2}+\lambda_iX^{1/4-\delta}\log^2X\,,\quad i=1,\,2\,.
\end{align}
We have
\begin{equation}\label{lambda12}
  \frac{\lambda_1}{\lambda_2}=\frac{\lambda_1t}{\lambda_2t}=
  \frac{\ds\frac{a_1}{q_1}+\bigg(\lambda_1t-\frac{a_1}{q_1}\bigg)}{\ds\frac{a_2}{q_2}+\bigg(\lambda_2t-\frac{a_2}{q_2}\bigg)}=
  \frac{a_1q_2}{a_2q_1}\cdot\frac{1+\mathfrak{T}_1}{1+\mathfrak{T}_2}\,,
\end{equation}
where $\mathfrak{T}_i=\dfrac{q_i}{a_i}\bigg(\lambda_it-\dfrac{a_i}{q_i}\bigg)\,,\; i=1,\,2$.
According to  (\ref{lambda1_a1_q1}), (\ref{lambda2_a2_q2}) and (\ref{lambda12}) we obtain
\begin{align}
  &|\mathfrak{T}_i|< \frac{q_i}{|a_i|}\cdot \frac{1}{q_iq_0^2}=\frac{1}{|a_i|q_0^2}\le \frac{1}{q_0^2}\,,\quad i=1,2\,,\nonumber\\
\label{lambd12}
  &\frac{\lambda_1}{\lambda_2}=\frac{a_1q_2}{a_2q_1}\cdot
  \frac{\ds 1+\Obig\bigg(\frac{1}{q_0^2}\bigg)}{\ds 1+\Obig\bigg(\frac{1}{q_0^2}\bigg)}=
 \frac{a_1q_2}{a_2q_1}\bigg(1+\Ob{\frac{1}{q_0^2}}\bigg)\,.\notag
\end{align}
Thus $\;\ds \frac{a_1q_2}{a_2q_1}=\Obig(1)$ and
\begin{equation}\label{lambd12new}
  \frac{\lambda_1}{\lambda_2}=\frac{a_1q_2}{a_2q_1}+\Obig\bigg(\frac{1}{q_0^2}\bigg)\,.
\end{equation}
Therefore, both fractions $\ds \frac{a_0}{q_0}$ and $\ds \frac{a_1q_2}{a_2q_1}$
approximate  $\ds \frac{\lambda_1}{\lambda_2}$.
Using (\ref{lambda1_a1_q1}),
(\ref{qi_logx}) and inequality
(\ref{ai}) with $i=2$ we obtain
\begin{equation}\label{a2q1}
  |a_2|q_1<1+\lambda_iX^{5/12-\delta}\log^2X<\frac{q_0}{\log X}\,
\end{equation}
so $|a_2|q_1\ne q_0$ and the fractions $\ds \frac{a_0}{q_0}$ and $\ds \frac{a_1q_2}{a_2q_1}$
are different. Then using (\ref{a2q1}) it follows
\begin{equation}\label{Contra}
  \bigg|\frac{a_0}{q_0}-\frac{a_1q_2}{a_2q_1}\bigg|=
  \frac{|a_0 a_2q_1-a_1q_2q_0|}{|a_2|q_1q_0}\ge \frac{1}{|a_2|q_1q_0}>\frac{\log X}{q_0^2}\,.
\end{equation}
On the other hand, from (\ref{q0}) and (\ref{lambd12new}) we have
\begin{equation*}
  \bigg|\frac{a_0}{q_0}-\frac{a_1q_2}{a_2q_1}\bigg|\le \bigg|\frac{a_0}{q_0}-\frac{\lambda_1}{\lambda_2}\bigg|+
  \bigg|\frac{\lambda_1}{\lambda_2}-\frac{a_1q_2}{a_2q_1}\bigg|\ll \frac{1}{q_0^2}\,,
\end{equation*}
which contradicts (\ref{Contra}).
This rejects the assumption (\ref{qi_logx}). Let
$q_0^{(1)},\,q_0^{(2)},\,\ldots$ be an infinite sequence
of values of $q_0$, satisfying (\ref{q0}).
Then using (\ref{X}) one gets an infinite sequence
$X_1,\,X_2,\,\ldots $ of values of $X$, such that
at least one of the numbers
$\lambda_{1}t$ and $\lambda_{2}t$
can be approximated by
rational numbers with denominators, satisfying (\ref{Int_q}).
Hence, the proof is completed.
\end{proof}

Let us estimate the integral
 $\Gamma_{1}^{(2)}(X_j)$, denoted by (\ref{Gamma1,2}).
Using (\ref{VtX}), (\ref{Vtxj}) and Lemma \ref{Fourier} we find
\begin{align}\label{Gamma1,2est1}
\Gamma_{1}^{(2)}(X_j)&\ll\vartheta\int\limits_{\tau<|t|<H}V(t,\,X_j)
\left(\left|L^{-}(\lambda_1 t,\,X_j)L^{+}(\lambda_3 t,\,X_j)\right|\right.
+\left.\left|L^{+}(\lambda_2 t,\,X_j)L^{+}(\lambda_3 t,\,X_j)\right|\right)dt\nonumber\\
&\ll\vartheta\int\limits_{\tau<|t|<H}V(t,X_j)
\left(\left|L^{-}(\lambda_1 t,\,X_j)\right|^2\right.
+\left.\left|L^{+}(\lambda_2 t,\,X_j)|^2+|L^{+}(\lambda_3 t,\,X_j)\right|^2\right)\,dt\nonumber\\
&\ll\vartheta X_j^{11/12+\varepsilon}\max\limits_{1\le k\le 3}\mathcal{I}_k\,,
\end{align}
where
\begin{equation*}
 \mathcal{I}_k=\int\limits_{\tau}^H\left|L^{\pm}(\lambda_k t,\,X_j)\right|^2\,dt.
 \end{equation*}
Arguing as in \cite{DimTod-1} we obtain
\begin{equation}\label{Ikest}
\mathcal{I}_k\ll X_j^{13/12-\delta}(\log X_j)^7\,.
\end{equation}
Using (\ref{Gamma1,2est1}), (\ref{Ikest}) and choosing $\varepsilon<\delta$ we get
\begin{equation}\label{Gamma1,2est2}
\Gamma_{1}^{(2)}(X_j)\ll\vartheta X_j^{11/12+\varepsilon}X_j^{13/12-\delta}(\log X_j)^7\ll
\vartheta \frac{X_j^2}{(\log X_j)^{A-5}}\,.
\end{equation}
Summarizing (\ref{Gama1}), (\ref{Gamma1,3est}), (\ref{Ga1,1}) and (\ref{Gamma1,2est2}) we find
\begin{equation}\label{Gama1est}
\Gamma_1(X_j)=B(X_j)G^-(G^+)^2+\mathcal{O}\left(\vartheta \frac{X_j^2}{(\log X_j)^{A-5}}\right)\,.
\end{equation}
\section{Proof of the Theorem.}
\indent

Since the sums $\Gamma_2(X_j)$, $\Gamma_3(X_j)$ and $\Gamma_4(X_j)$  are estimated in the same
way then from (\ref{Gamma}), (\ref{Gammaest1}), (\ref{Gammawaveest}), (\ref{Gamma0}) and
(\ref{Gama1est}) we obtain
\begin{equation}\label{Gammaest2}
\Gamma(X_j)\geq B(X_j)W(X_j)+\mathcal{O}\left(\vartheta \frac{X_j^2}{(\log X_j)^{A-5}}\right)\,,
\end{equation}
where
\begin{equation}\label{W}
W(X_j)=3\left(G^+\right)^2\left(G^--\frac{2}{3}G^+\right)
\end{equation}
and $G^\pm$ are defined by \eqref{Gpm}.\\
Let $f(s)$ and $F(s)$ are the lower and the upper functions of the linear sieve.
We know that if
\begin{equation}\label{s}
    s=\frac{\log D}{\log z},\quad 2\leq s\leq3
\end{equation}
then
\begin{equation}\label{Eli}
f(s)=\frac{2e^\gamma\log(s-1)}{s}\,, \qquad F(s)=\frac{2e^\gamma}{s}
\end{equation}
where $\gamma=0.577...$ is the Euler constant (see Lemma 10,\cite{Bru}).
Using (\ref{Gpm}) and Lemma 10 [1] we get
\begin{align}\label{G_iF_i}
  \mathcal{F}(z)\bigg(f(s)+&\mathcal{O}\left((\log X)^{-1/3}\right)\bigg)\nonumber\\
  &\quad\quad\quad\le G^{-} \le \mathcal{F}(z)
  \le G^{+}\nonumber\\
  &\quad\quad\quad\quad\quad\quad\quad\le \mathcal{F}(z)\bigg (F(s)+\mathcal{O}\big((\log X)^{-1/3}\big)\bigg )\,.
\end{align}
Here
\begin{equation}\label{piz}
    \mathcal{F}(z)=\prod\limits_{2<p\le z}\bigg(1-\frac{1}{p-1}\bigg)
   \asymp\frac{1}{\log X}\,.
\end{equation}
To estimate $W(X_j)$ from below we shall use the inequalities (see (\ref{G_iF_i})):
\begin{align}
\begin{split}\label{G1theta1}
   G^--\frac{2}{3}G^+&\geq \mathcal{F}(z)\bigg (f(s)-\frac{2}{3}F(s)+\mathcal{O}\big((\log X)^{-1/3}\big)\bigg )\\
   G^+&\geq \mathcal{F}(z)\,.
\end{split}
\end{align}
Let $X=X_j$. Then from (\ref{W}) and (\ref{G1theta1}) it follows
\begin{equation}\label{WepsXj}
W(X_j)\geq 3\mathcal{F}^3(z)\bigg (f(s)-\frac{2}{3}F(s)+\mathcal{O}\big((\log X)^{-1/3}\big)\bigg)
\end{equation}
Choose $s=2.948$.\\
Then by (\ref{z}), (\ref{D}) and (\ref{s}) we find
\begin{equation*}
\beta=0.035089.
\end{equation*}
It is not difficult to compute that for sufficiently large $X$ we have
\begin{equation}\label{fF}
  f(s)-\frac{2}{3}F(s)>10^{-5}.
\end{equation}
Choose $A\geq10$.\\
Then by (\ref{eps}), (\ref{estB}), (\ref{Gammaest2}), (\ref{piz}), (\ref{WepsXj}) and  (\ref{fF}) we obtain:
\begin{equation}\label{GammaXj}
  \Gamma(X_j)\gg \frac{X_j^{23/12+\delta}}{(\log X_j)^3}\,.
\end{equation}
The last inequality implies that $\Gamma(X_j) \rightarrow\infty$ as $X_j\rightarrow\infty$.

By the definition (\ref{Gamma}) of $\Gamma(X)$ and the inequality (\ref{GammaXj}) we conclude that for some
constant $c_0>0$ there are at least $c_0\frac{X_j^{23/12+\delta}}{(\log X_j)^6}$ triples of primes $ p_1, p_2, p_3$
satisfying $\lambda_0X_j< p_1, p_2, p_3\le X_j,\;|\lambda_1p_1+\lambda_2p_2+\lambda_3p_3+\eta|<\vartheta$
and such that for every prime factor $p$ of  $p_j+2,\,j=1,2,3$ we have $ p\geq X^{0.035089}$.

The proof of the Theorem is complete.

\vskip20pt
\footnotesize
\begin{flushleft}
S. I. Dimitrov\\
Faculty of Applied Mathematics and Informatics\\
Technical University of Sofia \\
8, St.Kliment Ohridski Blvd. \\
1756 Sofia, BULGARIA\\
e-mail: sdimitrov@tu-sofia.bg\\
\end{flushleft}

\end{document}